\documentclass[10pt]{article}
\usepackage{amsmath}
\usepackage{amsfonts}
\usepackage{amssymb}
\usepackage{amsthm}
\usepackage[margin=2.5cm, vmargin={1.5cm},includefoot]{geometry}

\usepackage[utf8]{inputenc}
\usepackage[T1]{fontenc}
\usepackage{graphicx}
\usepackage{longtable}
\usepackage{hyperref}
\usepackage{pst-node}
\usepackage{pst-coil}
\usepackage{pst-plot}
\usepackage{pst-grad}
\usepackage{pstricks}

\title{Classical curves via one-vertex maps}
\author{Anthony Weaver\\Department of Mathematics and Computer Science\\
Bronx Community College\\City University of New York\\
212-568-0669\\
anthonyweaver@mac.com
}
\date{May 1, 2012}

\begin{document}

\maketitle

\newtheorem{theorem}{Theorem}[section]
\newtheorem{lem}[theorem]{Lemma}
\newtheorem{prop}[theorem]{Proposition}
\newtheorem{cor}[theorem]{Corollary}
\newtheorem{conj}{Conjecture}
\newtheorem{rem}{Remark}
\newtheorem{dfn}{Definition}

\vspace*{1cm}

\begin{abstract} One-vertex maps (a type of {\it dessin d'enfant}) 
 give a uniform characterization of certain well-known algebraic curves, including those of Klein, Wiman, Accola-Maclachlan and Kulkarni.  The characterization depends on a new classification of one-vertex (dually, one-face or unicellular) maps according to the size of the group of map automorphisms.  We use an equivalence relation appropriate for studying the faithful action of the absolute Galois group on dessins, although we do not pursue that line of inquiry here.
\end{abstract}

\section{Introduction}

Klein's quartic curve, $w^3z + z^3 + w = 0$, admits a tesselation by 336 hyperbolic triangles, and has $\text{PSL}_2(7)$ as its full group of conformal  automorphisms  \cite{K79}.  It is characterized as the curve of smallest genus  realizing the upper bound $84(g-1)$  on the order of a group of conformal  automorphisms of a curve of genus $g>1$, given by A. Hurwitz \cite{H93} in 1893.      Around the same time, A. Wiman characterized the curves $w^2 = z^{2g+1}-1$ and $w^2=z(z^{2g}-1)$, $g>1$,   
as the unique curves of genus $g$ admitting cyclic automorphism groups of largest and second-largest  possible order ($4g+2$ and $4g$, respectively) \cite{W1895}.      Two other curve families, identified much later, share a similar characterization:  the Accola-Maclachlan and Kulkarni curves of genus $g$  realize the sharp lower bound $8g+8$ on the maximum order of an automorphism group of a curve of genus $g>1$ (\cite{A68}, \cite{M69}, \cite{K91}).   Collectively, we call these curves (together with one other exceptional curve discovered by Wiman) the {\it classical curves}.   Further details are given in Section~\ref{S:classical}.  

The triangular tesselation on the Klein quartic is an example of a {\it map}, or imbedding of a finite connected graph on a compact oriented  surface,  so that  the complement  is a union of simply connected faces.    Maps are a special type of {\it dessin d'enfant} ("child's drawing") as defined by G. Grothendieck \cite{G84}.   The topological surface underlying a map has a canonical complex structure,  making it a compact Riemann surface, and hence,  a complex algebraic curve (see, e.g., \cite{S57}, \cite{JS87}).    It turns out that any curve with a map (or dessin) admits a Bely\u\i\  function (a meromorphic function with at most three critical values) and is therefore, by Bely\u\i's theorem (\cite{B80}, \cite{K04}, \cite{W97}),  defined over an algebraic number field.    Conversely, any curve defined over a number field admits a map which is geometric in the sense that its edges are geodesics in a canonical complex structure \cite{JS96}.   The absolute Galois group acts faithfully on maps (or dessins) via its action on the coefficients of the equations defining the curves and the formulae for the  Bely\u\i\  functions \cite{JS97}.    

In this paper, we show that the classical curves can be uniformly characterized as curves admitting a strictly edge-transitive one-vertex map whose automorphisms are a proper subgroup of the full group of conformal automorphisms of the curve (see Section~\ref{SS:regular} for definitions).  The main result is Theorem~\ref{T:main}.  The paper includes three expository sections (Sections~\ref{S:combinatorics}, \ref{S:conformal} and \ref{S:extendability}) and two sections containing new results (Sections~\ref{S:onevertex} and \ref{S:largeonevertex}).      Section~\ref{S:combinatorics} gives the definition of a map and its automorphisms purely combinatorially in terms of permutation groups.  Section~\ref{S:conformal} constructs the canonical complex structure on a surface with a map, making the map geometric and its automorphisms  conformal.    Section~\ref{S:extendability} discusses conformal group actions on Riemann surfaces, and the question of whether and how  a finite cyclic action can extend to the action of a larger group.  Details on the classical curves, all of which admit large cyclic automorphism groups,  are also given there.   In Section~\ref{S:onevertex} we enumerate equivalence classes of one-vertex maps according to the size of the group of map automorphisms,  staying within the purely combinatorical framework.         In Section~\ref{S:largeonevertex} we specialize to regular and strictly edge-transitive one-vertex maps, whose automorphism groups are, respectively, maximal and second-maximal (with respect to the number of edges).   The classical curves (and some others) arise when the  second-maximal  group of map automorphisms is not the full group of conformal automorphisms of the curve (Theorem~\ref{T:main}).

One-face or unicellular maps are  the duals of one-vertex maps, and  have been studied from various points of view.   They were enumerated by means of a recurrence relation on the number of maps of a given genus  \cite{HZ86}, and more recently by a  different combinatorial identity \cite{Ch11}.   Neither enumeration gives any explicit information about the automorphism groups, as ours does.    In \cite{S01}, D. Singerman characterized Riemann surfaces admitting strictly edge-transitive  uniform unicellular dessins, a less general class of maps than the one we focus on.  One-face maps of genus $0$ (plane trees) have been extensively studied (see, e.g., \cite{JS97}, \cite{LZ04} and references therein), since they are easily visualized, and the action of the  absolute Galois group, restricted to these simple dessins, remains faithful.    The final paragraph of the paper summarizes the work most closely related to our main result.

The recent book \cite{LZ04} is an excellent general reference on dessins; see also \cite{JS96}.     The first complete exposition of the theory of maps on surfaces was given by G.A.  Jones and D. Singerman \cite{JS78}, and it remains an excellent reference.  The paper~\cite{MP98} describes an analytic approach to maps using quadratic differentials.

\section{Maps and permutation groups}\label{S:combinatorics}

A {\it map}  is  a finite connected graph embedded on a compact oriented  surface so that the complement is a union of simply connected {\it faces}.   Loops and multiple edges are allowed.   A directed edge is called a {\it dart}, symbolized by an arrowhead on the edge pointing toward one of the two incident vertices.   We usually assume that all edges carry two darts, although it is possible and sometimes useful to relax this assumption.      The set of  darts  pointing toward a given vertex receives a counter-clockwise cyclic ordering from the orientation of the ambient surface.  If the map  has $k$ edges, and the darts are labelled arbitrarily with the symbols $0,1, \dots, 2k-1$, the cyclic ordering of the darts at each vertex gives a collection of disjoint cycles in the symmetric group $S_{2k}$; we define  $x \in S_{2k}$ to be the product of these cycles.  We define $y \in S_{2k}$ to be the permutation that interchanges the dart labels on each edge.  With the assumption that all edges carry two darts, $y$ is a free involution.    The {\it monodromy group} $G_{\cal M}$ of a map ${\cal M}$ is  the subgroup $ \langle x,y \rangle \leq S_{2k}$.  By the connectedness of the underlying graph, $G_{\cal M}$ acts transitively on the darts.  
 
 Figure~\ref{F:torusmap} shows a map with $k=24$ darts imbedded on the torus  of modulus $e^{\pi i/3}$.   (The lower left corner of the parallelogram is  $0 \in {\mathbb C}$ and the upper left is $e^{\pi i /3}$.) There are $8$ vertices, $12$ edges, and $4$ hexagonal faces.  Evidently 
 \begin{gather}\notag
 x=(0\ 1\ 2)\ (3\ 4\ 5)\ (6\ 7\ 8)\ (9\ 10\ 11)\ (12\ 13\ 14)\ (15\ 16\ 17)\ (18\ 19\ 20)\ (21\ 22\ 23), \\ \notag
 y=(0\ 10)\ (1\ 17)\ (2\ 3)\ (4\ 6)\ (5\ 13)\ (7\ 23)\ (8\ 9)\ (11\ 19)\ (12\ 22)\ (14\ 15)\ (16\ 18)\ (20\ 21) \in S_{24}.
 \end{gather}
 We will revisit this example throughout the paper.
 
 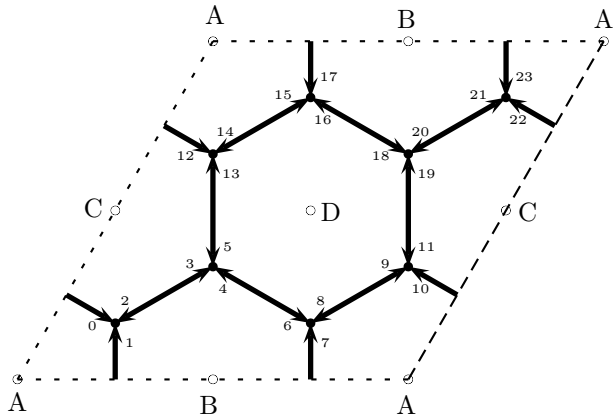
\begin{figure}
 \begin{center}
 \psset{unit=1.5}
\begin{pspicture}(0,0)(8,4)
\SpecialCoor
\degrees[6]
\rput{5.5}(0,0)
{
 

  
  
  
   \rput(1;0){\psdots[dotstyle=o,dotscale=1](1;2)}
   \rput(2;0){\psdots[dotstyle=o,dotscale=1](3;2)}
  \rput(3;0){\psdots[dotstyle=o,dotscale=1](2;2)(5;2)}
     \rput(4;0){\psdots[dotstyle=o,dotscale=1](4;2)}
   \rput(5;0){\psdots[dotstyle=o,dotscale=1](3;2)(6;2)}
      \rput(6;0){\psdots[dotstyle=o,dotscale=1](5;2)}
    \rput(7;0){\psdots[dotstyle=o,dotscale=1](7;2)}
    
    \psdots[dotstyle=*,dotscale=1](2;1)(6;1)
    
    \rput(3;0){\psdots[dotstyle=*,dotscale=1](3;2)(4;2)}
    \rput(4;0){\psdots[dotstyle=*,dotscale=1](3;2)(5;2)}
     \rput(5;0){\psdots[dotstyle=*,dotscale=1](4;2)(5;2)}
    

\rput(1;1){\psline[linestyle=dashed, dash= 2pt  5pt, dash= 2pt  5pt](0,0)(!0 3 sqrt 2 mul)}   
\rput(2;0){\rput(3;1){\psline[linestyle=dashed](0,0)(!0 3 sqrt 2 mul)}}

\rput{.5}(1;1){\psline[linestyle=dashed, dash= 2pt  5pt](0,0)(!3 sqrt 2 mul 0)} 

\rput{3.5}(7;1){\psline[linestyle=dashed, dash= 2pt  5pt](0,0)(!3 sqrt 2 mul 0)} 

\rput(2;1){\psline[linewidth=2pt, arrows=<->](0,0)(1;1)}
\rput(5;1){\psline[linewidth=2pt, arrows=<->](0,0)(1;1)}

\rput{1}(3;1){\psline[linewidth=2pt, arrows=<->](0,0)(1;1)}
\rput{5}(3;1){\psline[linewidth=2pt, arrows=<->](0,0)(1;1)}

\rput{2}(5;1){\psline[linewidth=2pt, arrows=<->](0,0)(1;1)}
\rput{4}(5;1){\psline[linewidth=2pt, arrows=<->](0,0)(1;1)}

\rput(1;0){\rput(3;1){\psline[linewidth=2pt, arrows=<->](0,0)(1;1)}}
\rput(-1;0){\rput(4;1){\psline[linewidth=2pt, arrows=<->](0,0)(1;1)}}

\rput{3}(2;1){\psline[linewidth=2pt, arrows=<-](0,0)(.5,0)}
\rput{5}(2;1){\psline[linewidth=2pt, arrows=<-](0,0)(.5,0)}
   \rput{0}(6;1){\psline[linewidth=2pt, arrows=<-](0,0)(.5,0)}
\rput{2}(6;1){\psline[linewidth=2pt, arrows=<-](0,0)(.5,0)}

\rput(1;0){\rput{0}(4;1){\psline[linewidth=2pt, arrows=<-](0,0)(.5,0)}}
\rput(3;0){\rput{3}(4;2){\psline[linewidth=2pt, arrows=<-](0,0)(.5,0)}}

\rput(4;0){\rput{5}(3;2){\psline[linewidth=2pt, arrows=<-](0,0)(.5,0)}}
\rput(4;0){\rput{2}(5;2){\psline[linewidth=2pt, arrows=<-](0,0)(.5,0)}}

}
\uput{7 pt}[3.1](2;.5){\tiny 0}
\uput{7 pt}[5.1](2;.5){\tiny 1}
\uput{7 pt}[1.1](2;.5){\tiny 2}

\uput{7 pt}[2.9](3;.5){\tiny 3}
\uput{7 pt}[4.9](3;.5){\tiny 4}
\uput{7 pt}[0.9](3;.5){\tiny 5}

\uput{7 pt}[3.1](5;.5){\tiny 18}
\uput{7 pt}[5.1](5;.5){\tiny 19}
\uput{7 pt}[1.1](5;.5){\tiny 20}

\uput{7 pt}[2.9](6;.5){\tiny 21}
\uput{7 pt}[4.9](6;.5){\tiny 22}
\uput{7 pt}[0.9](6;.5){\tiny 23}

\rput(4; .5){
\uput{7 pt}[3.1](1;4.5){\tiny 6}
\uput{7 pt}[5.1](1;4.5){\tiny 7}
\uput{7 pt}[1.1](1;4.5){\tiny 8}
\uput{7 pt}[2.9](1;1.5){\tiny 15}
\uput{7 pt}[4.9](1;1.5){\tiny 16}
\uput{7 pt}[0.9](1;1.5){\tiny 17}

\uput{7 pt}[2.9](1;5.5){\tiny 9}
\uput{7 pt}[4.9](1;5.5){\tiny 10}
\uput{7 pt}[0.9](1;5.5){\tiny 11}
\uput{7 pt}[3.1](1;2.5){\tiny 12}
\uput{7 pt}[5.1](1;2.5){\tiny 13}
\uput{7 pt}[1.1](1;2.5){\tiny 14}

}

\uput[4.5](1;.5){A}
\rput(1;.5){\uput[3](1.75;1){C}}
\rput(1;.5){\uput[1.5](3.5;1){A}}  

\uput[1.5](7;.5){A}
\rput(7;.5){\uput[0](1.75;4){C}}
\rput(7;.5){\uput[4.5](3.5;4){A}}  

\rput(7;.5){\uput[1.5](1.75;3){B}}
\rput(7;.5){\rput(1,75;3){\uput[0](1.75;4){D}}}
\rput(7;.5){\rput(1,75;3){\uput[4.5](3.5;4){B}}} 	

\end{pspicture}

 \caption{A map on a torus.  Opposite sides of the parallelogram are identified.}\label{F:torusmap}
\end{center}
\end{figure}
 
 For any map with monodromy group $\langle x, y \rangle$, the cycles  of $(xy)^{-1}=yx^{-1}$ describe closed oriented circuits (paths of darts arranged tip-to-tail) which bound the faces of the map.     To see this, start with a  dart $\delta_1$ pointing toward a given vertex $v_1$.   $x^{-1}\delta_1$ is the next dart after $\delta_1$ in the {\it clockwise} cyclic ordering of the darts pointing toward $v_1$.  Let $\delta_2 = yx^{-1}\delta_1$,   the ``reversal" of $x^{-1}\delta_1$, pointing along the same edge,  toward an adjacent vertex, $v_2$.   (It may happen that $v_2 = v_1$;  in this case the edge carrying $\delta_2$  is a loop at $v_1$).  Repeat the process staring with $\delta_2$, obtaining $\delta_3 =y x^{-1}\delta_2$,  pointing toward an adjacent vertex, $v_3$ (possibly $v_3 = v_2$ or $=v_1$).   After finitely many steps, say, $n\geq 1$ of them,   $\delta_n=\delta_1$, that is, we arrive again at the original dart.    At no step does an  edge incident with a vertex $v_i$ intervene between the edge carrying  $\delta_i$ and the edge carrying  $\delta_{i+1}$.  Thus the  circuit $\delta_1, \delta_2, \dots, \delta_{n-1}$ traverses the boundary of a unique face of ${\cal M}$, namely, the face which is on the left from the point of a view of a traveller following the circuit.   With this convention, each dart belongs to the boundary of a unique face.    For the map in Figure~\ref{F:torusmap},
 $$yx^{-1}=(0\ 3\ 13\ 22\ 20\ 11)\ (1\ 10\ 8\ 23\ 12\ 15)\ (2\ 17\ 18\ 21\ 7\ 4)\ (5\ 6\ 9\ 19\ 16\ 14) \in S_{24}.$$
  The four cycles specify the boundaries of the hexagonal faces centered at $C$, $A$, $B$, $D$, respectively.
  
 The topological genus $g$ of the surface underlying a map is easily determined from the formula for the Euler characteristic, $2g-2 =| \{\text{vertices}\}| -|\{\text{edges}\}| + |\{\text{faces}\}|$, where the number of vertices  is   the number of cycles in $x$, and the number of faces is  the number of cycles in $(xy)^{-1}$ ($1$-cycles are counted).  The number of edges is the number of $2$-cycles in a free involution on $2k$ symbols, that is, $k$.

     If there is a permutation in $S_{2k}$ which {\it simultaneously} conjugates the monodromy generators $x_1,y_1 \in G_{{\cal M}_1}$ of a map ${\cal M}_1$ to, respectively,  the monodromy generators $x_2,y_2 \in G_{{\cal M}_2}$ of another map ${\cal M}_2$,   we say that the groups $G_{{\cal M}_1}$, $G_{{\cal M}_2}$ are {\it strongly conjugate}.
     \begin{dfn} Two maps  ${\cal M}_1$, ${\cal M}_2$ are {\em equivalent} if their monodromy groups are strongly conjugate. 
       \end{dfn}
    \noindent  A map $\cal M$ is (non-trivially) equivalent to itself if and only if there is a permutation in $S_{2k}$,  not contained in $G_{\cal M}$,  which commutes with both monodromy generators and hence centralizes $G_{\cal M}$.    
     Since conjugate subgroups have conjugate centralizers, the {\it automorphism group\,} of a map,  
$$\text{Aut}({\cal M})=\text{Cent}_{S_{2k}}(G_{\cal M}) = \{\pi \in S_{2k} \mid \pi g = g \pi \text{\ for all\ } g \in G_{\cal M}\},$$
 is well-defined (up to isomorphism) on equivalence classes  of maps.   
     
            The equivalence relation we adopt is  appropriate  for studying the faithful action of the absolute Galois group on maps:   if maps from distinct equivalence classes  are in the same Galois orbit, their monodromy groups are conjugate but not strongly conjugate (see, e.g., \cite{JS97}).  We do not pursue this since we focus on maps on surfaces defined over ${\mathbb Q}$, which are fixed by the Galois action.  
              
\subsection{Regular and edge-transitive maps}\label{SS:regular}
   \begin{dfn} A map ${\cal M}$ has {\em type} $(n,r)$, if $n$ is the lcm  (least common multiple) of the vertex valencies and $r$ is the lcm of the face  valencies.  (The valence of a face is the number of darts in its bounding circuit.)  The map is {\em uniform} if all vertices have valence equal to $n$ and all faces have valence equal to $r$.
   \end{dfn} 

   The map ${\cal M}$ is called {\it regular} if $\text{Aut}({\cal M})$ acts transitively on the set $D$ of darts.    This implies that the map is uniform, and in particular,  that every dart has the same local incidence relations.     The map in Figure~\ref{F:torusmap} is uniform of type $(3,6)$.  It is also regular: $\text{Aut}({\cal M}) \simeq {\mathbb Z}_6 \ltimes ({\mathbb Z}_2 \oplus {\mathbb Z}_2)$, where  ${\mathbb Z}_2 \oplus {\mathbb Z}_2$ is the factor group generated by translations carrying $A\mapsto B$ and $A\mapsto C$, modulo the normal subgroup generated by their squares.    The study of regular maps goes back at least to Klein and Dyck, and perhaps as far as Kepler (see \cite{CM72}, Chapter 8); for more recent work see, e.g., \cite{C09}, \cite{JSW10}.
   
   A map for which $\text{Aut}({\cal M})$ is  transitive on edges (but not necessarily on darts) is called {\it edge-transitive}.  In this case there are at most two sets of local incidence relations since a dart might not be equivalent (via an automorphism) to its reversal.   Thus there are at most two distinct face-valences, and at most two distinct vertex valences.    We call a map which is edge-transitive but not regular  {\it strictly\,} edge transitive ("half-regular" in \cite{S01}). 
 

\section{Uniformization of maps and surfaces}\label{S:conformal}

  Let $G=G_{\cal M}=\langle x, y \rangle$ be  the monodromy group of a map ${\cal M}$ of type $(n,r)$.  Let $\Gamma=\Gamma(n,r)$ be the  group with presentation 
\begin{equation}\label{E:gammapres}
\Gamma(n,r)= \langle \xi_1,\xi_2, \xi_3 \mid \xi_1^n = \xi_2^2 = \xi_3^{r} = \xi_1\xi_2\xi_3=1 \rangle.
\end{equation}
There is an obvious surjective homomorphism $
\theta: \Gamma \rightarrow G$, defined by  
\begin{equation}\label{E:theta}
 \theta: \xi_1 \mapsto x, \quad \xi_2 \mapsto y,\quad \xi_3 \mapsto (xy)^{-1}.
 \end{equation}
For $\delta \in D$,  where $D$ is the set of darts of ${\cal M}$, let $G_\delta$ denote the isotropy subgroup $\{g \in G \mid \delta g = \delta\}$.  
 \begin{dfn}\label{D:mapsubgp} 
The subgroup $M = \theta^{-1}(G_\delta)\leq \Gamma$,  is called the {\rm canonical map subgroup} for ${\cal M}$.
\end{dfn}
\noindent Transitivity of the permutation group $(G, D)$  implies all $G_\delta$ are conjugate, so $M$ is well-defined (up to conjugacy) independent of $\delta$.     The {\it core} of $M$ in $\Gamma$ is the normal subgroup
$$M^\ast = \bigcap_{\gamma \in \Gamma}\gamma^{-1}M\gamma.$$
$\Gamma/M^\ast$ acts on the set $| \Gamma/M |$ of  cosets  $M$ in $\Gamma$, as follows:   
 $M^*\gamma_1:  M\gamma \mapsto M\gamma\gamma_1 $, $ \gamma,\gamma_1 \in \Gamma.$  
 One may verify that the action is well-defined, faithful and transitive.

 \begin{lem}\label{L:mapsubgroup}
 The permutation groups $(G, D)$ and $(\Gamma/M^*,| \Gamma/M |)$ are isomorphic.
\end{lem}
\begin{proof} It suffices to verify that the diagram \begin{equation}\notag
\begin{matrix} G &\times& D& \rightarrow & D \\
\overset{i \big\downarrow}{\ } & & \overset{b \big\downarrow}{\ } && \overset{b \big\downarrow}{\ } \\
\displaystyle{\frac{\Gamma}{M^\ast}} &\times & \displaystyle{\left| \frac{\Gamma}{M} \right|}&\rightarrow  &\displaystyle{\left| \frac{\Gamma}{M} \right|}
\end{matrix}
\end{equation}
is commutative, where arrows denote group actions, and the maps $i$ and $b$ are defined as follows:  $i: G \rightarrow \Gamma/M^\ast$ maps $g \mapsto M^\ast\theta^{-1}(g)$;   for fixed but arbitrary $\delta \in D$,   $b: D \rightarrow | \Gamma/M |$  maps  $b:\delta g \mapsto   M\theta^{-1}(g)$, $g \in G$.  To show that $i$  is a well-defined isomorphism, one first verifies that $M^\ast = \text{ker}(\theta)$; then $i$ is simply the inverse of the canonical isomorphism $\Gamma / \text{ker}(\theta) \rightarrow G$.  $b$ is a well-defined bijection according to the following argument, whose steps are reversible:    for $g_1, g_2 \in G$, $\delta g_1 =\delta  g_2  \iff g_2g_1^{-1} \in G_\delta$ $\iff \theta^{-1}(g_2g_1^{-1}) \subseteq M$  $\iff M\theta^{-1}(g_2)(\theta^{-1}(g_1))^{-1} = M$ $\iff M\theta^{-1}(g_1) = M\theta^{-1}(g_2)$.  
\end{proof}
 
 \subsection{Groups with signature}  
 
 The map subgroup $M$, and its overgroup $\Gamma$,  are examples of  {\it groups with signature}.  These groups act  properly discontinuously by conformal isometries on  one of the three simply connected Riemann surfaces:   the Riemann sphere (${\mathbb P}^1$), the complex plane (${\mathbb C}$), or the Poincar\'e upper half plane (${\mathbb H}$).  This fact allows us to obtain a canonical complex structure on the topological surface underlying a map.

A group $\Lambda=\Lambda(\sigma)$ with {\it signature} $\sigma=(h; r_1, r_2, \dots, r_s)$ has presentation
\begin{equation}\label{E:sigpres}
\Lambda(\sigma)=\langle a_1, b_1, \dots, a_h, b_h, \xi_1, \dots \xi_s \mid \xi_1^{r_1} = \dots = \xi_s^{r_s} = \prod_{i=1}^ha_i^{-1}b_i^{-1}a b \prod_{j=1}^s \xi_j =1 \rangle,   
\end{equation}  
and acts by conformal isometries on  ${\mathbb P}^1$,  ${\mathbb C}$, or  ${\mathbb H}$,  depending on whether the number
$$\mu(\sigma) = 2h-2 + \sum_{i=1}^s\bigl(1-\frac{1}{r_i}\bigr)$$
is, respectively, negative, $0$, or positive.      If $s>0$, the $r_i$  are called the {\it periods} of  $\Lambda$.  All periods are $>1$, and the number of occurrences of a given period is the number of conjugacy classes of elements of that order in $\Lambda$.  Signatures which differ only by a permutation of the periods determine isomorphic groups and are considered the same.   If $s=0$, $\Lambda$ is a torsion-free {\it surface group}, with signature $(h; -)$, isomorphic to the fundamental group of a surface of genus $h$.
  If $\Delta$  is any subgroup of finite index $n$ in $\Lambda(\sigma)$,  then $\Delta$ is also a group with signature $\sigma'$ and   the {\it Riemann-Hurwitz relation} states that  $\mu(\sigma') = n \mu(\sigma)$.

  The orbit space  $\,{\cal U}/\Lambda$, where $\,{\cal U}$ denotes ${\mathbb P}^1$,  ${\mathbb C}$, or  ${\mathbb H}$ as appropriate,  is a compact Riemann surface of genus $h$ with $s$ distinguished  points, over which the projection $\,{\cal U} \rightarrow {\cal U}/\Lambda$ branches.    The complex structure on ${\cal U}/\Lambda$ is the one which makes the branched covering map $\,{\cal U} \rightarrow {\cal U}/\Lambda$ holomorphic.     The uniformization theorem of Klein, Poincar\'e and Koebe states that  every compact surface of genus $h$ can be  obtained in this way; moreover one may choose the uniformizing group $\Lambda$ to be a surface group, so that there are no distinguished points, and the covering is regular.    
  
    When $\mu(\sigma)>0$, which is the case except for finitely many signatures, $\Lambda(\sigma)$ is a co-compact {\it Fuchsian group} (discrete subgroup of $\text{PSL}(2,{\mathbb R})$ containing no parabolic elements) and $2\pi\mu(\sigma)$ is the hyperbolic area of a fundamental domain for the action of $\Lambda(\sigma)$ on ${\mathbb H}$.   The normalizer $N_{\text{PSL}(2,{\mathbb R})}(\Lambda)$ of a Fuchsian group $\Lambda$  is itself a Fuchsian group, containing $\Lambda$ with finite index.  
    \begin{lem}\label{L:confaut} The automorphism group of  $\,{\mathbb H}/\Lambda$  is isomorphic to the finite quotient group $N_{\text{PSL}(2, {\mathbb R})}(\Lambda)/\Lambda$.
    \end{lem}
 For proofs and further details, see, e.g.,  \cite{JS78},  \cite{JS87}, \cite{SK92}.

\subsection{The canonical Riemann surface of a map}\label{SS:universalmaps}

Let $M \leq \Gamma=\Gamma(n,r)$ be the canonical map subgroup of a map ${\cal M}$ of type $(n,r)$.     Since $M$ is a group with signature, the quotient space 
$\,{\cal U}/M$  is a compact Riemann surface, known as the {\em canonical Riemann surface} for ${\cal M}$.  The canonical Riemann surface contains  the map ${\cal M}$ {\it geometrically},  that is,  the edges of ${\cal M}$ are geodesics, faces are regular geodesic polygons,  and face-centers are well-defined points.   To see this,  we pass to the {\it universal map} $\hat{\cal M}_{n,r}$ of type $(n,r)$.  $\hat{\cal M}_{n,r}$ is a regular map  imbedded on one of  the simply connected Riemann surfaces $\,{\cal U}$ $={\mathbb P}^1$,  ${\mathbb C}$, or  ${\mathbb H}$.  It can be constructed as follows.  
Let $T \in {\cal U}$ be a geodesic triangle with vertices $a,b,c$,  at which the  interior angles are $\pi/n$, $\pi/2$, $\pi/r$ respectively (see Figure~\ref{F:universalmap}).   Reflections in the sides of $T$ generate a discrete group of isometries of $\,{\cal U}$ whose sense-preserving subgroup  (of index $2$) is isomorphic to $\Gamma=\Gamma(n,r)$.    The generators $\xi_1$, $\xi_2$, $\xi_3$ $\in \Gamma$ correspond to the rotations about  $a$, $b$, $c$ through angles of, respectively,  $2\pi/n$, $\pi$, $2\pi/r$;  the product of the three rotations  (considered as a product of six side reflections)  is easily seen to be trivial.    Let ${\bf \hat e}$ be the geodesic segment   $\overline{ab} \cup \xi_2(\overline{ab})$, directed toward the vertex $a$.  The map with dart set   $\hat D=\{\gamma({\bf \hat e}) \mid \gamma \in \Gamma\}$ and vertex set    $\hat V=\{\gamma(a) \mid \gamma \in \Gamma\}$ is the universal map $\hat{\cal M}_{n,r}$.  
   Darts meet in sets of $n$ at each vertex and the angle between any two consecutive edges is  $2\pi/n$.   The faces are regular geodesic $r$-gons centered at the $\Gamma$-images of $c$.      
 The permutation group $(\Gamma, \hat D)$ is isomorphic to $(\Gamma,  |\Gamma|)$, the  right regular representation of $\Gamma$ on itself.  Thus the map is regular and $\text{Aut}(\hat{\cal M}_{n,r}) = \text{Aut}(\Gamma, |\Gamma |) = \Gamma$.  
 
  If the number $\dfrac 12 - \dfrac 1n -\dfrac 1r$    is negative, $0$, or positive, $\hat{\cal M}_{n,r}$ lies on ${\mathbb P}^1$,  ${\mathbb C}$, or  ${\mathbb H}$, respectively.  In the negative case, $\hat{\cal M}_{n,r}$ is a  central projection of one of the Platonic solids onto the circumscribed sphere; in the positive case, it is one of infinitely many regular tesselations of the hyperbolic plane.  In the $0$ case, it is a tesselation of ${\mathbb C}$ by squares, hexagons, or equilateral triangles.  (Figure~\ref{F:torusmap} is a portion of $\hat{\cal M}_{3,6}$.)   See, e.g., \cite{CM72}.

 \begin{figure}
 \begin{center}
 \psset{unit=1.4}
 \begin{pspicture}(0,.5)(7,4)

\psarc(-3,0){5}{22.4}{37}
\psarc(1,0){2}{72}{90}
\psline(1,2)(1,3)
\uput[90](1,3){a}
\uput[180](1,2){b}
\uput[315](1.6, 1.9){c}

\uput[0](.8, 2.6){ \tiny $\dfrac{\pi}{n}$}
\uput[0](1.1,2.15){ \tiny $\dfrac{\pi}{r}$}
\pspolygon(1,2)(1.1,2)(1.1,2.1)(1, 2.1)

\uput[0](1.1, 1.4){$T$}

\psarc[linestyle=dashed](0,0){5}{22.4}{37}
\psarc[linestyle=dashed](4,0){2}{72}{90}
\psline[linewidth=1.5pt, arrows=->](4,1)(4,3)

\uput[90](4,3){a}
\uput[315](4.6, 1.9){c}
\uput[180](4,1){a$\xi_2$}

\psarc[linewidth=1.5pt, arrows=->](4,.5){2.5}{37}{90}
\uput[315](6, 2){a$\xi_2\xi_1$}


\uput[0](3.8, 2.6){ \tiny $\dfrac{\pi}{n}$}

\uput[0](4.15, 2.7){ \tiny $\dfrac{\pi}{n}$}

\pspolygon(4,2)(4.1,2)(4.1,2.1)(4, 2.1)

\uput[0](3.5,2){$\bf \hat e$}

\uput[0](5.2,2.8){${\bf \hat e} \xi_1$}

\psdots(4,1)(4,3)(6,2)

\psdots[dotstyle=o,dotscale=1](4.6, 1.9)
\end{pspicture}
\caption{Construction of $\hat{\cal M}_{n,r}$, the universal map of type $(n,r)$}\label{F:universalmap}
\end{center}
\end{figure}
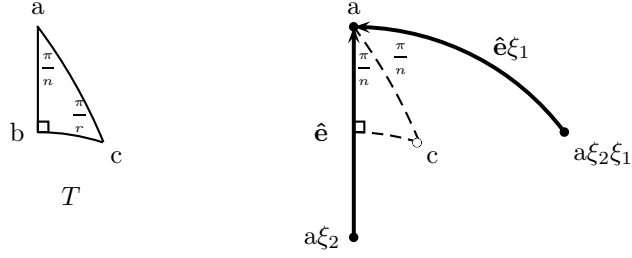

  The orbifold map ${\cal U} \rightarrow {\cal U}/M$ is a local isometry.   Hence    
the image of $\hat D$ in the canonical Riemann surface $\,{\cal U}/M$  is  a set of directed geodesic segments.  These segments are in bijection with the set $|\Gamma/M|$ of right cosets of $M$ in $\Gamma$, and   $\Gamma$ acts by right multiplication on $|\Gamma/M|$.   If $\gamma \in \Gamma$ fixes {\it every} coset, it belongs to  the core $M^\ast$ of $M$ in $\Gamma$.   Hence the universal map $\hat{\cal M}_{n,r}$ projects to a map $\,\hat {\cal M}_{n,r}/M$ on $\,{\cal U}/M$ -- the promised geometric map -- with  monodromy group    $(\Gamma/M^*,| \Gamma/M |)$.   With Lemma~\ref{L:mapsubgroup}, we conclude that {\it every} map of type $(n,r)$ is a quotient of the universal map $\,\hat{\cal M}_{n,r}$,  by a subgroup $M$ of finite index in the group $\Gamma=\Gamma(n,r)$, and may be assumed to lie on its canonical Riemann surface.

The general theory of covering spaces specializes to the category of maps of type $(n,r)$.  Thus, if ${\cal M}_1$ and ${\cal M}_2$ are two  finite  maps with canonical map subgroups $M_1, M_2 \leq \Gamma=\Gamma(n,r)$, then ${\cal M}_1$ is a finite covering of ${\cal M}_2$ if and only if  $M_1$ is conjugate within $\Gamma$ to a subgroup of finite index in $M_2$.     
   In particular, if $M_2=N_\Gamma(M_1)$, the normalizer of $M_1$ in $\Gamma$, then $\text{Aut}({\cal M}_1) \simeq M_2/M_1$.  This leads to the following Lemma.    

   \begin{lem}\label{L:mapaut} Let ${\cal M}$ be a map of type $(n,r)$ with canonical map subgroup $M \leq \Gamma = \Gamma(n,r)$.  Then $\text{Aut}({\cal M}) \leq \text{Aut}(X)$, where $X ={\cal U}/M$ is the canonical Riemann surface for ${\cal M}$,  and $\text{Aut}(X)$ is the group of conformal automorphisms of $X$.
    \end{lem}
  \begin{proof}   $\text{Aut}({\cal M})$ is isomorphic to  $N_\Gamma(M)/M$.  
 Since $\Gamma$, $M$ and $N_\Gamma(M)$ are discrete subgroups of $\text{Isom}^+({\cal U})$, the sense-preserving isometries of  ${\,\cal U}$,  $$ \frac{N_\Gamma(M)}{M} 
\leq \frac{N_{\text{Isom}^+({\cal U})}(M)}{M} =\text{Aut}({\,\cal U}/{M}).$$ 
The equality on the right is a consequence of Lemma~\ref{L:confaut} if ${\cal U}={\mathbb H}$.
  \end{proof}
  
      In the local geometry of the canonical Riemann surface, an automorphism $\alpha \in \text{Aut}({\cal M})$ which  preserves a  face acts   as a rotation about the fixed face-center through an angle $2\pi/r'$, where $r'$ is a divisor of the face-valence (and a divisor of $r$ if the type is $(n,r)$).     Similarly, if $\alpha$  preserves an edge while transposing its darts, the geometric action is a rotation through an angle $\pi$ about the fixed edge midpoint.     These are consequences  of corresponding facts about the  universal map $\hat{\cal M}_{n,r}$,  together with the fact that the covering projection ${\cal U} \rightarrow {\cal U}/M$ is a local isometry.  See \cite{JS78}.

 \section{One-vertex maps with automorphisms}\label{S:onevertex}
 
 We now enumerate one-vertex maps, up to equivalence, according to the size of their automorphism groups.

 Let ${\cal M}$ be a map with one vertex and $k$  edges.    The monodromy generator $x \in S_{2k}$  is a  $2k$-cycle specifying the counterclockwise cyclic ordering of the  darts (arbitrarily labelled $0,1,2, \dots, 2k-1$) around the vertex.  The other monodromy generator, $y\in S_{2k}
 $, is a free involution whose  $k$ disjoint transpositions specify how the darts pair off into edges.   
Since all $2k$-cycles are conjugate in $S_{2k}$,   we may assume, up to map equivalence,  that $x$ is the standard $2k$-cycle
\begin{equation}\label{E:standardcycle}
(0\ 1\ 2\   \dots\  2k-1).
\end{equation}
Then ${\cal M}$  is determined  
by the free involution $y$, 
which can be chosen in 
$$\prod_{i=0}^{k-1}\begin{pmatrix} 2k-2i \\2 \end{pmatrix} =  (2k-1)\cdot (2k-3)\cdot \dots \cdot 3 \cdot 1 = (2k-1)!!$$ 
ways.   Figure~3 shows three of the $15$ one-vertex maps with three edges.  

The number of equivalence classes of one-vertex maps with $k$ edges is smaller than $(2k-1)!!$,  since some of the maps have non-trivial automorphisms.   

\begin{theorem}\label{T:MapAuts} Let ${\cal M}$ be a one-vertex map with $k$ edges and monodromy group $G_{\cal M} = \langle x,y \rangle$,  where $x$ is a $2k$-cycle and  and $y$ is a free involution.  Then:  
(i) $\text{Aut}({\cal M}) \leq \langle x \rangle \simeq {\mathbb Z}_{2k}.$   (ii) If $\text{Aut}({\cal M})=\langle x^p \rangle$, $p$ a divisor of $2k$,  the monodromy groups
$\langle x, y_s \rangle,$  where $y_s=x^{-s}yx^s$, $s=0, \dots p-1$, determine distinct but equivalent maps.  Conversely, if ${\cal M}$ is equivalent to ${\cal M}^\prime$ with monodromy group $G_{{\cal M}^\prime }= \langle x, y^\prime \rangle$,  then $\text{Aut}({\cal M}) \simeq \text{Aut}({\cal M}^\prime) = \langle x^p \rangle$ for some divisor $p$ of $2k$, and $y^\prime = y_s$ for some $s$, $0 \leq s \leq p-1$.
 \end{theorem}
 \begin{proof}   The centralizer of $G_{\cal M}$  is contained in the centralizer of each of its generators;  in particular, it is contained in the centralizer of $\langle x \rangle$.  The latter is an abelian transitive permutation group which, by a well-known result in the theory of permutation groups, is its own centralizer (see, e.g., \cite{S87}, Section 10.3).    It follows that $\text{Aut}({\cal M}) \leq \langle x \rangle \simeq {\mathbb Z}_{2k}.$  Now suppose $p$ is a divisor of $2k$ and $\text{Aut}({\cal M})=\langle x^p \rangle$.   The permutations $x^s$, $s=1, \dots, p-1$ do not commute with $y$ (otherwise they would be automorphisms of ${\cal M}$), hence there are $p$ distinct free involutions $y_s=x^{-s}y{x^s}$,  $s=0, 1, \dots, p-1$,  belonging to strongly conjugate monodromy groups $\langle x, y_s \rangle$, determining distinct but equivalent maps.  Conversely, if ${\cal M}$ is equivalent to ${\cal M}^\prime$, the monodromy groups $G_{{\cal M}}= \langle x, y \rangle$ and $G_{{\cal M}^\prime }= \langle x, y^\prime \rangle$ are strongly conjugate, with conjugate centralizers and hence isomorphic automorphism groups.    Since both automorphism groups are contained in $\langle x \rangle$,  $\text{Aut}({\cal M}) = \text{Aut}({\cal M}^\prime)=\langle x^p \rangle$ for some divisor $p$ of $2k$.  By strong conjugacy of $G_{\cal M}$ and $G_{{\cal M}^\prime}$,  there exists a permutation $\sigma \in S_{2k}$, such that $\sigma^{-1}x\sigma = x$ and $\sigma^{-1}y\sigma = y'$.  In particular $\sigma \in \text{Cent}_{S_{2k}}(\langle x \rangle) = \langle x \rangle$.  If $y^\prime \ne y$, $\sigma \in \langle x \rangle \setminus \langle x^p \rangle$.  Thus $\sigma = x^d$, for $d$ a non-multiple of $p$.  If $d>p$,  $\sigma^{-1} y\sigma={\sigma^\prime}^{-1}y{\sigma^\prime}$, where $\sigma^\prime = x^{d-p}$, since $x^p$ commutes with $y$.   It follows that all possible $y^\prime$ are obtained by taking $\sigma \in \{x, x^2, \dots, x^{p-1}\}$.
 \end{proof}

\subsection{Free involutions with prescribed commutation property}\label{SS:prescribedcommuters}

 If $x \in S_{2k}$ is a $2k$-cycle, and $p$ is a divisor of $2k$, then $x^p$ is a product of $p$ cycles of length $d=2k/p$.  Assuming $x$ is the standard $2k$-cycle \eqref{E:standardcycle},  
 the cycles are
 \begin{equation}\label{E:p-cycles}
 C_j = (j\  j+p\  \dots\ j+2k-p),\qquad j=0, 1, \dots, p-1.
 \end{equation}
  Let $c_y: S_{2k} \rightarrow S_{2k}$  denote conjugation by a free involution $y \in S_{2k}$.   If  $y$ commutes with $x^p$,  then  $\langle c_y \rangle \simeq{\mathbb Z}_2$   block-permutes the $p$ $d$-cycles at \eqref{E:p-cycles},  while preserving their internal cyclic orderings.     The permutation $c_y$ either  preserves a given cycle  or interchanges two distinct cycles in a block-orbit of length $2$.     Suppose $c_y$ preserves the cycle  $C_j$, as well as its internal cyclic order.  Then for some positive integer $t$ and for every symbol $a \in\  C_j$, $c_y(a) = a+tp \pmod{2k}$.  Because $c_y$ has order $2$,  $2tp \equiv 0 \pmod{2k}$ or, equivalently, 
  $tp \equiv 0 \pmod k$, which has the unique (mod $k$) solution $t=k/p=d/2$, implying that {\it $d$ is even}.      On the other hand, 
  if $c_y: S_{2k}\rightarrow S_{2k}$  interchanges two cycles $C_i$ and $C_j$, $i<j$, preserving their internal cyclic orders,  there exists an integer $t \geq 0$ such that  every symbol $i+ap \in\  C_i$, is mapped to $c_y(i+ap) = j+(a+t)p$, and, reciprocally,  every symbol  $j + bp \in \ C_j$, is mapped to $c_y(j+bp) = (i + (b-t)p$ (all symbols are reduced mod $2k$).  The {\it shift parameter} $t$ is not uniquely determined, but may take any value $0, 1, \dots, d-1$.  There is no requirement that $d$ be even.   It follows that a free involution commuting with $x^p$ is uniquely determined by the following data:
  \begin{itemize}
  \item  the set of cycles $C_j$ fixed by $c_y$ (empty if $d=2k/p$ is odd);
  \item the (possibly empty) set of cycle pairs  $[C_i, C_j]$, $i<j$, transposed by $c_y$;
  \item for each transposed pair, a choice of shift parameter $t$, $0 \leq t \leq d-1$.
  \end{itemize}
Let $q=\lfloor p/2\rfloor$, the greatest integer less than or equal to $p/2$.  If $c_y$ transposes $m \leq q$ cycle pairs, those pairs can be chosen, up to reordering, in
$$\dfrac{1}{m!}\begin{pmatrix}p \\ 2 \end{pmatrix} \begin{pmatrix}p-2 \\ 2 \end{pmatrix}  \dots \begin{pmatrix}p-(2m-2) \\ 2 \end{pmatrix}= \biggl(\frac 12 \biggr)^m \frac{p!}{m!(p-2m)!}$$
ways.    Each pair requires a choice of shift parameter  from among $0,1,2,\dots, d-1$, so the number of choices is multiplied by $d^m$.  
If $d$ is odd, there are no unpaired cycles and $m=p/2=q$, an integer since $p$ is even.   If $d$ is even,  unpaired cycles are possible so the number of choices is a sum over all possible $m$.  This yields the following Lemma.  
\begin{lem}\label{L:freecommuters} Let $\overline{\nu}_p$ be the number of free involutions in $S_{2k}$ commuting with $x^p$, where $x$ is a $2k$-cycle and  $p$ a divisor of $2k$.   Then
\begin{equation}\notag
\overline{\nu}_p = 
\begin{cases} \displaystyle{\sum_{m=0}^{ q } \biggl(\frac d2 \biggr)^m \frac{p!}{m!(p-2m)!}} &\text{if $d$ is even} \\
\displaystyle{\biggl(\frac d2 \biggr)^q\  \frac{p!}{q!} }&\text{if $d$ is odd},
\end{cases}
\end{equation}
where $d=2k/p$, and  $q=\lfloor \frac p2 \rfloor$. 
\end{lem}
\noindent 
The Lemma yields $\overline{\nu}_{2k} = (2k-1)!!$,  as it should, since every free involution commutes with $x^{2k}=\{1\}$, and $\overline{\nu}_{1} = 1$, since  $x^k$ is the unique (free) involution in $\text{Cent}_{S_{2k}}(x) = \langle x \rangle$.

\subsection{One-vertex maps with prescribed automorphism}

If $y$ commutes with $x^p$, but no lower power of $x$, then 
$\langle x^p \rangle$ is the {\it full} automorphism group of a one-vertex map whose monodromy group is strongly conjugate to $\langle x, y \rangle$.

\begin{theorem}\label{T:fullautcount} Let $\nu_p$ be the number of one-vertex maps with $k$ edges whose full automorphism group is $\langle x^p \rangle$, $p$ a divisor of $2k$.   Then  
$\nu_p = \overline{\nu}_p + \sum_{i=1}^s (-1)^i\sigma_i$,
where 
$$\sigma_1 = \sum_{1 \leq j \leq s}\overline{\nu}_{p/p_j},\quad \sigma_2 = \sum_{1\leq j <k \leq s}\overline{\nu}_{p/p_jp_k}, \quad \sigma_3 = \sum_{1 \leq j<k<l \leq s}\overline{ \nu}_{p/p_jp_kp_l},  \quad \dots, $$
and $p_i$, $i=1, \dots, s$, are the prime divisors of $p$.
\end{theorem}
\begin{proof}
 Let $p'$ be any proper divisor of $p$.  If $p$ is a prime, $p'=1$ and $\overline{\nu}_1=1$.  Hence $\nu_p = \overline{\nu}_p-1 = \overline \nu_p-\sigma_1$, proving the theorem in this case.  If $p$ is not prime, $p=p' p_1^{\epsilon_1}p_2^{\epsilon_2}\dots p_s^{\epsilon_s}$, where  at least one of the $\epsilon_i > 0$.  Reordering the $p_i$ if necessary we may assume that the first $t\leq s$ of the $\epsilon_i > 0$, while the remaining $s-t$ are $0$.  Hence,  without loss of generality, $p=p'p_1^{\epsilon_1}p_2^{\epsilon_2}\dots p_t^{\epsilon_{t}}$,  with all $\epsilon_i > 0,$ $i=1,2, \dots, t$.  
  Each map with automorphism group $\langle x^{p'} \rangle$ contributes $1$ to $\overline\nu_p$; it suffices to show that such a map contributes $-1$ to$ \sum_{i=1}^s (-1)^i\sigma_i$, so that its total contribution to the formula in the theorem is $0$.  We use a standard inclusion/exclusion argument.   If $t=1$, the map contributes $-1$ to $-\sigma_1$ and $0$ to the remaining summands.  If $t=2$, the map contributes $-2$ to $-\sigma_1$, $1$ to $\sigma_2$, and $0$ to the remaining summands.  In general, for arbitrary $t\leq s$, the contribution is $-(\begin{smallmatrix} t \\ 1\end{smallmatrix})$ to $-\sigma_1$,  $(\begin{smallmatrix} t \\ 2\end{smallmatrix})$ to $\sigma_2$, etc.  It follows from the binomial theorem that the total contribution is
  $$1 + \sum_{i=1}^t(-1)^i\begin{pmatrix}t \\ i\end{pmatrix}=(1-1)^t=0.$$  Thus the formula counts only those maps whose automorphism group is equal to $\langle x^p \rangle$.   
  \end{proof}
 \begin{cor}\label{C:equivcount} The number of equivalence classes of one-vertex maps with $k$ edges whose full automorphism group is  $\langle x^p \rangle$, $p$ a divisor of $2k$, is $\nu_p/p$.
\end{cor}
\begin{proof}  Theorem~\ref{T:MapAuts}(ii) combined with Theorem~\ref{T:fullautcount}. 
  \end{proof}

  In particular, the number of equivalence classes of strictly edge-transitive,  one-vertex maps with $k$ edges is
\begin{equation}\label{E:edgetranscount}
\nu_2/2 = \begin{cases}
k/2 &\text{if $k$ is even;}\\
(k-1)/2 &\text{if $k$ is odd.}
\end{cases}
\end{equation}
Since $\nu_1 = 1$, there is a unique regular one-vertex map with $k$ edges.  Figure~\ref{F:k=3} illustrates the case $k=3$.

\begin{figure}
\begin{center}
\begin{pspicture}(0,-2)(13,2.5)
\SpecialCoor
\degrees[6]
\rput(1;0){\psdots(1;1)}
\rput(6;0){\psdots(1;1)}
\rput(10;0){\psdots(1;1)}
\rput(1;0){\psarc[linewidth=0.04](1;1){1}{0}{1}}
\rput(1;0){\psarc[linewidth=0.04](1;1){1}{2}{3}}
\rput(1;0){\psarc[linewidth=0.04](1;1){1}{4}{5}}
\rput(1;0){\rput(1;1){\psline[linewidth=1.5 pt, arrows=<-](0,0)(1;0)}}
\rput(1;0){\rput(1;1){\psline[linewidth=1.5 pt, arrows=<-](0,0)(1;1)}}
\rput(1;0){\rput(1;1){\psline[linewidth=1.5 pt, arrows=<-](0,0)(1;2)}}
\rput(1;0){\rput(1;1){\psline[linewidth=1.5 pt, arrows=<-](0,0)(1;3)}}
\rput(1;0){\rput(1;1){\psline[linewidth=1.5 pt, arrows=<-](0,0)(1;4)}}
\rput(1;0){\rput(1;1){\psline[linewidth=1.5 pt, arrows=<-](0,0)(1;5)}}

\rput(6;0){\psarc[linewidth=0.04](1;1){1.2}{0}{3}}
\rput(6;0){\psarc[linewidth=0.04](1;1){1}{1}{4}}
\rput(6;0){\psarc[linewidth=0.04](1;1){.8}{2}{5}}
\rput(6;0){\rput(1;1){\psline[linewidth=1.5 pt, arrows=<-](0,0)(1.2;0)}}
\rput(6;0){\rput(1;1){\psline[linewidth=1.5 pt, arrows=<-](0,0)(1.2;3)}}
\rput(6;0){\rput(1;1){\psline[linewidth=1.5 pt, arrows=<-](0,0)(1;1)}}
\rput(6;0){\rput(1;1){\psline[linewidth=1.5 pt, arrows=<-](0,0)(1;4)}}
\rput(6;0){\rput(1;1){\psline[linewidth=1.5 pt, arrows=<-](0,0)(0.8;2)}}
\rput(6;0){\rput(1;1){\psline[linewidth=1.5 pt, arrows=<-](0,0)(0.8;5)}}

\rput(10;0){\psarc[linewidth=0.04](1;1){1}{1}{2}}
\rput(10;0){\psarc[linewidth=0.04](1;1){1}{3}{4}}
\rput(10;0){\psarc[linewidth=0.04](1;1){1}{5}{6}}
\rput(10;0){\rput(1;1){\psline[linewidth=1.5 pt, arrows=<-](0,0)(1;1)}}
\rput(10;0){\rput(1;1){\psline[linewidth=1.5 pt, arrows=<-](0,0)(1;2)}}
\rput(10;0){\rput(1;1){\psline[linewidth=1.5 pt, arrows=<-](0,0)(1;3)}}
\rput(10;0){\rput(1;1){\psline[linewidth=1.5 pt, arrows=<-](0,0)(1;4)}}
\rput(10;0){\rput(1;1){\psline[linewidth=1.5 pt, arrows=<-](0,0)(1;5)}}
\rput(10;0){\rput(1;1){\psline[linewidth=1.5 pt, arrows=<-](0,0)(1;6)}}

\rput(1;0){
\uput{9 pt}[0.4](1;1){\tiny 0}
\uput{9 pt}[1.4](1;1){\tiny 1}
\uput{9 pt}[2.4](1;1){\tiny 2}
\uput{9 pt}[3.4](1;1){\tiny 3}
\uput{9 pt}[4.4](1;1){\tiny 4}
\uput{9 pt}[5.4](1;1){\tiny 5}
}

\rput(6;0){
\uput{9 pt}[0.4](1;1){\tiny 0}
\uput{9 pt}[1.4](1;1){\tiny 1}
\uput{9 pt}[2.4](1;1){\tiny 2}
\uput{9 pt}[3.4](1;1){\tiny 3}
\uput{9 pt}[4.4](1;1){\tiny 4}
\uput{9 pt}[5.4](1;1){\tiny 5}
}

\rput(10;0){
\uput{9 pt}[0.4](1;1){\tiny 0}
\uput{9 pt}[1.4](1;1){\tiny 1}
\uput{9 pt}[2.4](1;1){\tiny 2}
\uput{9 pt}[3.4](1;1){\tiny 3}
\uput{9 pt}[4.4](1;1){\tiny 4}
\uput{9 pt}[5.4](1;1){\tiny 5}
}

\uput[0](-.5, -.6){$x= (0\ 1\ 2\ 3\ 4\ 5)$}
\uput[0](-.5,-1){$y=  (0\ 1)(2\ 3)(4\ 5)$}
\uput[0](-.5, -1.4){$yx^{-1} = (0\ 4\ 2)\ (1)\ (3)\ (5)$}

\uput[0](4.5, -.6){$x= (0\ 1\ 2\ 3\ 4\ 5)$}
\uput[0](4.5,-1){$y=  (0\ 3)(1\ 4)(2\ 5)$}
\uput[0](4.5, -1.4){$yx^{-1} = (0\ 2\ 4)\ (1\ 3\ 5)$}

\uput[0](9, -.6){$x= (0\ 1\ 2\ 3\ 4\ 5)$}
\uput[0](9,-1){$y=  (1\ 2)(3\ 4)(5\ 0)$}
\uput[0](9, -1.4){$yx^{-1} = (0)\ (2)\ (4)\ (1\ 3\ 5)$}

\uput[0](-.8,1.6){(a)}
\uput[0](4,1.6){(b)}
\uput[0](8.6,1.6){(c)}
\end{pspicture}
\caption{One-vertex maps with three edges.  (a) and (c) are equivalent and strictly edge-transitive;  (b) is regular.}\label{F:k=3}
 \end{center}
 \end{figure}
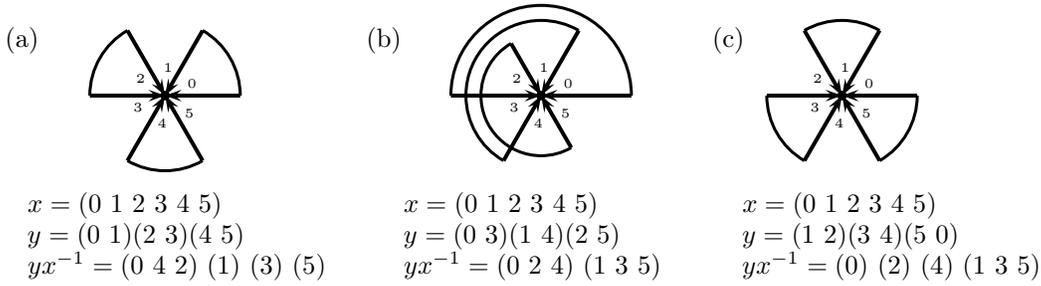
  
 \section{Conformal automorphisms of Riemann surfaces}\label{S:extendability}
 
 A finite group G acts by conformal automorphisms on a compact Riemann surface $X$ of genus $g$ if and only if  there is a group $\Lambda(\sigma)$, with signature  $\sigma$,  containing a normal surface subgroup $M_g$, with signature $(g; -)$,  and a {\it surface-kernel epimorphism\,} $\rho: \Lambda \rightarrow G$ making the sequence of homomorphisms
 \begin{equation}\label{E:shortexact}
 \{1\} \rightarrow M_g \hookrightarrow \Lambda(\sigma) \overset{\rho}{\rightarrow} G \rightarrow \{1\}
 \end{equation}
 exact (see \cite{M61}, or \cite{JS87}, Theorem 5.9.5).   The surface $X$ is conformally equivalent to ${\,\cal U}/M_g$, where ${\,\cal U}$ is the universal covering space.   $G$ is isomorphic to $\Lambda(\sigma)/M_g$,  and the Riemann-Hurwitz relation requires $2g-2=|G|\mu(\sigma)$.    We say that {\it $G$ acts in genus $g$ with signature $\sigma$}.     
 \begin{dfn}  A {\em $\sigma$-generating vector  for a finite group $G$} is  an ordered set of elements, one for each generator of $\Lambda(\sigma)$,   which generate $G$ and fulfill  the corresponding relations in $\Lambda(\sigma)$ (and possible other relations).  
 \end{dfn}
\noindent A $\sigma$-generating vector for $G$ determines a surface kernel epimorphism $\rho$ satisfying \eqref{E:shortexact}, and conversely.  Hence 
\begin{lem}$G$ acts in genus $g$ with signature $\sigma$ if and only if $G$ has a $\sigma$-generating vector.
\end{lem}

 Let  $\sigma=(0;n, m, r)$, so that
$\Lambda(\sigma)$ is a {\it triangle group\,}, and let $G={\mathbb Z}_n$.  By a general theorem of Maclachlan \cite{M65} (see also Harvey \cite{H66}),  an abelian group of order $n$ has an $(0;n,m,r)$-generating vector  if an only if $\text{lcm}(m,r)=n$.  We establish a normal form for such a generating vector.  Let ${\mathbb Z}_n = \langle x \mid x^n=1 \rangle$, and let  $\langle x^a, x^b, x^c \rangle$ be a $(0;n,m,r)$-generating vector.  Then $a+b+c \equiv 0 \pmod n$ since the product of the elements must be the identity.   For the elements to have the appropriate orders,  $n/(a,n) =n$, $n/(b,n) = m$, and $n/(c,n) = r$, where $(\ ,\ )$ denotes the gcd (greatest common divisor) of two integers.    If $\rho$ in \eqref{E:shortexact} is post-composed with an automorphism of $G$, an equivalent generating vector, specifying a topologically equivalent action, is obtained (see \cite{B90}, Prop. 2.2).  Hence we may assume $a=1$.   Writing ${\mathbb Z}_n$  additively, with the generator  $1$,  we define 
\begin{dfn} The {\em normal form} for an {\it $(0;n,m,r)$-generating vector\,} for ${\mathbb Z}_n = \langle \{0,1,\dots, n\}, + \rangle$ is  
 \begin{equation}\label{E:ZnGenVector}
 \langle\  1,\  b,\  c\  \rangle, \quad 1+b+c \equiv 0 \ (\text{mod\ }n), \quad n/(b,n)=m, \quad n/(c,n)=r.
 \end{equation}
 \end{dfn}

\subsection{Finite maximality and extendability}
 \begin{dfn} A group with signature is {\em finitely maximal} if it is not a subgroup of finite index in any other group with signature.
 \end{dfn}
\noindent    If $\Lambda(\sigma)$  in \eqref{E:shortexact} is finitely maximal, then $\Lambda(\sigma)/M_g$ is necessarily the {\it full\,} automorphism group of $X={\,\cal U}/M_g$. 
    
      Finite maximality is not generally a property of the abstract group $\Lambda(\sigma)$, but depends on the particular imbedding of the group in $\text{Isom}^+({\cal U})$.  However,  there are some pairs of signatures $(\sigma, \sigma')$ such that every imbedding of a group $\Lambda(\sigma)$ with signature $\sigma$ in $\text{Isom}^+({\cal U})$ extends to an imbedding of an overgroup $\Lambda(\sigma')$.  Signature pairs of this type were considered by Greenberg in \cite{G74}, and fully classified by Singerman \cite{S72}.    For each such pair, there arises a delicate problem in finite group theory:   for each subgroup-overgroup pair $(G,G')$, with $[G':G]=[\Lambda(\sigma'):\Lambda(\sigma)]$, and such that $G'$ has a $\sigma'$-generating vector and $G$ has a $\sigma$-generating vector, determine conditions under which the $G'$-action is an extension of the $G$-action {\em on the same surface}.  Equivalently, determine conditions under which the surface-kernel epimorphism $\rho$ in \eqref{E:shortexact} is the restriction of a surface-kernel epimorphism (with the same kernel) $\rho':\Lambda(\sigma') \rightarrow G'$.  This problem is treated comprehensively in the papers 
  \cite{BCC02} and \cite{BC99}.   The signature pairs $(\sigma, \sigma')$ where both signatures are triangular and the first is {\it cyclic-admissable} -- that is, $\Lambda(\sigma)$ admits surface-kernel epimorphisms onto a cyclic $G$ -- are given in Table~\ref{Ta:cyclicadmissible}, which uses the case nomenclature established in \cite{BCC02}.   The index $[\Lambda(\sigma'):\Lambda(\sigma)]=[G':G]$ is  easily computed using the Riemann-Hurwitz relation and the fact that $M_g$ is a subgroup of both groups.  In cases N6 and N8 in the table, $\Lambda(\sigma)$ is a normal subgroup of  $\Lambda(\sigma')$.   We give complete descriptions of these cases below.  The necessary numerical conditions on  $k$ are implicit but not stated in \cite{BCC02}.
  \begin{table}[h]
\begin{center}
\begin{tabular}{| r | c | c | c | c |}
\hline
Case & $\sigma$ & $\sigma'$ & $[\Lambda(\sigma') :\Lambda(\sigma)]$& Conditions\\
\hline
N6 & $(0;k,k,k)$ & $(0;3,3,k)$ & $3$& $k\geq 4$ \\
N8 & $(0;k,k,u)$ & $(0;2,k,2u)$ & $2$ &$u | k$, $k\geq 3$ \\
T1 & $(0;7,7,7)$ &$(0;2,3,7)$ & $24$& -\\
T4 & $(0;8,8,4)$ & $(0;2,3,8)$ & $12$& - \\
T8 & $(0;4k,4k,k)$ & $(0;2,3,4k)$ & $6$ & $k\geq 2$ \\
T9 & $(0;2k,2k, k)$ & $(0;2,4,2k)$ & $4$ &$k\geq 3$  \\
T10 & $(0;3k,k,3)$ & $(0;2,3,3k)$ & $4$&$k\geq 3$  \\
\hline
\end{tabular}
\end{center}
\caption{Cyclic-admissible signatures ($\sigma$) and possible extensions ($\sigma'$)}\label{Ta:cyclicadmissible}
\end{table}

\begin{lem}[Case N6]\label{L:N6}
 A ${\mathbb Z}_k$ action with signature $(0;k,k,k)$  has an extension of type {\rm N6} to a $G$ action with signature $(0;3,3,k)$  only if $3^2 \nmid k$ and $p \equiv 1 \pmod 3$ for every prime divisor $p\ne 3$ of $k$.   The numerical conditions imply that ${\mathbb Z}_k$ has an automorphism $\beta$ of order $3$, acting non-trivially on each $p$-Sylow subgroup, $p \ne 3$.  Then $G \simeq {\mathbb Z}_3 \ltimes_\beta {\mathbb Z}_k$ and the normal form of the generating vector for the ${\mathbb Z}_k$ action is $\langle 1, \beta(1), \beta^2(1)\rangle$.
\end{lem}
\begin{proof}  Since ${\mathbb Z}_k$ is a normal subgroup, $G$ has the semi-direct product structure ${\mathbb Z}_3 \ltimes_{\beta}{\mathbb Z}_k$, where $\beta$ is a (possibly trivial) automorphism of ${\mathbb Z}_k$ of order a divisor of $3$.  In addition, 
 any non-trivial subgroup of $G$, being  a homomorphic image of the group with signature $(0;3,3,k)$, is generated by at most two elements of order $3$.  Let $^pG$ denote a $p$-Sylow subgroup of $G$.   $^3G$ has the structure ${\mathbb Z}_3 \ltimes_{\beta}{\mathbb Z_{3^s}}$, $s \geq 0$, but since it is generated by a most two elements of order $3$, $s \leq 1$.   Hence $3^2 \nmid k
 $.  For a prime divisor $p\ne 3$  of $k$, $^pG$ is cyclic, and $\beta$ induces an automorphism $\beta_p$ on $^pG$.  The subgroup ${\mathbb Z}_3 \ltimes_{\beta_p}(^pG) \leq G$  is  generated by two elements of order $3$ only if  $\beta_p$ is a non-trivial automorphism of order $3$.  This implies that $p \equiv 1 \pmod 3$, and yields the stated necessary conditions on $k$.    If $1$ is the additive generator of $^pG$,  then the element $1 + \beta_p(1) + \beta_p^2(1) \in ^pG$ is fixed  by $\beta_p$ and hence must be the identity (equivalently, $1 + \beta_p(1) + \beta_p^2(1) \equiv 0 \pmod {|^pG|}$).  Now,   ${\mathbb Z}_k$ is  the direct sum of its $p$-Sylow subgroups; hence $\beta = \oplus_p \beta_p$,  where the sum is over the prime divisors of $k$.  ($\beta_p$ is trivial if $p=3$.)  It follows that $\langle \beta \rangle$ is a subgroup of order $3$ in the automorphism group of ${\mathbb Z}_k$ which acts non-trivially on each $p$-Sylow subgroup, $p \ne 3$.  Taking  $1$ as the additive generator of ${\mathbb Z}_k$,  we have $1+\beta(1) + \beta^2(1) \equiv 0 \pmod k$.   Thus $\langle 1, \beta(1), \beta^2(1) \rangle$ is the normal form of a $(0;k,k,k)$ generating vector for a ${\mathbb Z}_k$ action extendable to a $G$ action.
\end{proof}

\begin{lem}[Case N8]\label{L:N8}
 A ${\mathbb Z}_k$ action with signature $(0;k,k,u)$ has an  extension of type {\rm N8} to a $G$ action with signature $(0;2,k,2u)$ if and only if either
 \begin{enumerate} 
 \item $k$ is odd, $G\simeq  {\mathbb Z}_{2k}$,  $u=k$, and the normal form of the ${\mathbb Z}_k$ generating vector is $\langle 1, 1, k-2 \rangle$; 
 \item $k$ is even, $G={\mathbb Z}_2 \oplus {\mathbb Z}_k$,  $u=k/2$,  and the normal form of the ${\mathbb Z}_k$ generating vector is $\langle 1, 1, k-2 \rangle$; 
\item $k \ne 2,4, p^s, 2p^s$, ($p$ an odd prime),  $G \simeq {\mathbb Z}_2\ltimes_{\alpha}  {\mathbb Z}_k$,  (non-abelian, non-dihedral),  where $\alpha$ is an automorphism of  ${\mathbb Z}_k$ of order $2$,  $\alpha(1) \ne -1$, and the normal form of the ${\mathbb Z}_k$ generating vector is $\langle 1, \alpha(1), k-(1+\alpha(1))\rangle$.  
 \end{enumerate}
 \end{lem}
 \begin{proof} The group $G$ has the semi-direct product structure ${\mathbb Z}_2 \ltimes_{\alpha}{\mathbb Z}_k$, where $\alpha$ is a (possibly trivial) automorphism of ${\mathbb Z}_k$ of order a divisor of $2$.   If $G$ is non-abelian, $\alpha$ is non-trivial and the generating vector is of the form $\langle 1, \alpha(1), c \rangle$.   $\alpha (1) \ne -1$, for then $c=0$, and $u = 1$.  Such an $\alpha$ exists if and only if $k \ne 2,4, p^s, 2p^s$, ($p$ an odd prime) \cite{IR82}.  $G$ has the (non-dihedral) semi-direct product structure ${\mathbb Z}_2\ltimes_{\alpha}  {\mathbb Z}_k$.  If  $G$ is abelian, $\alpha$ is trivial, and by Maclalchlan's lcm condition \cite{M65}, either $u=k$ and $G \simeq {\mathbb Z}_{2k}$,  or $u=k/2$ and $G={\mathbb Z}_2 \oplus {\mathbb Z}_k$.   
 \end{proof}

\subsection{Cyclic actions on the classical curves}\label{S:classical}

The classical curves admit automorphism groups which are ``large'' with respect to the genus $g$;  in particular, the maximal cyclic subgroups of automorphisms have order at least $\geq 2g+1$.   
The paper \cite{K97} gives a  precise definition of ``large,''  and has   interesting points of contact with our work.   (See, e.g., Section~3 of \cite{K97}, where cyclic group actions with a fixed point are described in terms of  side-pairings of hyperbolic polygons; being one-face maps, the polygons can be viewed as the duals of one-vertex maps in which  side-pairings correspond to dart-pairings on edges.)

The Wiman  curves of genus $g>1$ are
\begin{gather}\notag
w^2 = z^{2g+1}-1 \qquad\text{(type  I)} \\ \notag
w^2=z(z^{2g}-1) \qquad\text{(type  II)} \\ \notag
w^3 = z^4 + 1 \quad(g=3) \qquad\text{(type III)}.
\end{gather}
Their maximal cyclic automorphism groups  are  
\begin{gather}\notag
(w,z) \mapsto (-w, ze^{2\pi i/2g+1})\quad \text{of order $4g+2$}\quad\text{(type I)} \\ \notag
(w,z) \mapsto (we^{\pi i/2g}, ze^{\pi i/g}) \quad \text{of order $4g$}\quad \text{(type II)}\\ \notag
(w,z) \mapsto (we^{2\pi i/3}, ze^{2\pi i/4}) \quad \text{of order $4g=12$} \quad\text{(type III)}.
\end{gather}
  Signatures for the cyclic actions can be deduced by analyzing the ramifications of the branched covering projection   
   $P_z: (w,z) \mapsto z$.    For example, for the Wiman curve of type I,  $P_z: (w,z) \mapsto z$ has degree $4g+2$ with ramification of order $2g+1$ over $z=0$, order $4g+2$ over $z=\infty$, and order $2$ over the $(2g+1)$st  roots of unity.   The unique normal form of a generating vector for ${\mathbb Z}_{4g+2}$ with a $(0;4g+2, 2g+1, 2)$ signature is easily seen to be $\langle 1, 2g, 2g+1\rangle$.  By similar arguments one sees that the signatures of the Type II and III actions are $(0;4g,4g,2)$ and $(0;12,4,3)$, respectively.   Corresponding generating vectors, unique in normal form,  are given in Table~\ref{Ta:largecyclic}.   The type I signature is finitely maximal (it does not appear in Table~\ref{Ta:cyclicadmissible}), so ${\mathbb Z}_{4g+2}$ is the full automorphism group of the type I curve.  The  type II and type III signatures are not:  the cyclic actions extend to larger conformal actions according to cases N8 and T10 of Table~\ref{Ta:cyclicadmissible}, with $k=4g$, $k=4$, respectively.    The full\footnote{If $g=2$,  $SD_{16}$ extends  to an action of $\text{GL}_2({\mathbb Z}_3)$, of order $48$,  with signature $(0; 2,3,8)$ (Case T11 in \cite{BCC02}).} automorphism groups are  
  \begin{gather}\notag
  \text{SD}_{8g} = \langle a,b \mid  a^{4g}=b^2 = 1,\quad b^{-1}ab = a^{2g-1} \rangle\quad\text{acting with signature $(0;2,4g,4)$}\quad\text{(type II)}\\ \notag
  \text{H}_{48} = \langle\text{Kulkarni's group of order $48$}\rangle \quad  \text{acting with signature $(0;2,3,12)$}\quad\text{(type III)}.
  \end{gather}
A presentation of $H_{48}$ is given in \cite{K97}.  

  In the late 1960's R.D.M. Accola \cite{A68} and C. Maclachlan \cite{M69} independently showed that the genus $g$ curve with equation  
    $w^2=z^{2g+2}-1$ 
          has full automorphism group 
           \begin{equation}\notag
   \text{AM}_{8g+8} = \langle a, b \mid a^{2g+2}=b^4=1, \quad (ab)^2=[a,b^2]=1\rangle,  
   \end{equation}
   of order $8g+8$, and that, for infinitely many $g \geq 2$, this curve realizes the maximal order of an automorphism group of a surface of genus $g$.   (A. M. Macbeath had shown in 1961 \cite{M61a} that the Hurwitz bound $84(g-1)$ is not attained for  infinite sequences of genera.)        
       Thus, for every $g \geq 2$, there exists a surface $X_g$ of genus $g$ such that $|\text{Aut}(X_g)| \geq 8g+8$, and the bound is sharp for infinitely many $g$.  
  If  $g \equiv -1 \pmod 4$, 
  there is a second surface of genus $g$, identified by Kulkarni in 1991 \cite{K91}, having full automorphism group of order $8g+8$, acting with the same signature as $\text{AM}_{8g+8}$.    The full\footnote{If $g=3$,  $K_{32}$ extends to an action by a group of order $96$ acting with signature $(0;2,3,8)$ (Case T11 in \cite{BCC02}).}  automorphism group of the Kulkarni surface is
 \begin{equation}\notag
 \text{K}_{8g+8} = \langle a, b \mid a^{2g+2}=b^4=1, \quad (ab)^2=1,\quad b^2ab^2=a^{g+2}\rangle.  
   \end{equation}
  Both $\text{AM}_{8g+8}$ and $\text{K}_{8g+8}$   act with signature $(0;2,4, 2g+2)$ and contain the cyclic subgroup  $\langle a \mid a^{2g+2}=1\rangle$ of index $4$. 
  The distinct normal forms of the  generating vectors for the cyclic actions, given in Table~\ref{Ta:largecyclic}, show that the Accola-Maclachlan and Kulkarni curves, when they exist in the same genus, are distinct.    
 The index $4$ extensions to $\text{AM}_{8g+8}$ and $\text{K}_{8g+8}$ realize Case T9 in Table~\ref{Ta:cyclicadmissible}. 

  The Klein quartic is uniquely determined by the normal form of the generating vector for its maximal cyclic group of automorphisms (${\mathbb Z}_7$) given in Table~\ref{Ta:largecyclic}.  This action has an index $3$ extension (type N6) to the group ${\mathbb Z}_3 \ltimes_\beta {\mathbb Z}_7$, $\beta(1)=2$, and a further extension (type T1) to   $\text{PSL}_2(7)$.  It is worth noting that there is another ${\mathbb Z}_7$ action in genus $3$, with the same signature $(0;7,7,7)$ as the one determining the Klein quartic.  This action has generating vector $\langle 1,1,5 \rangle$ and determines the Wiman curve of type I via an N8 extension, with $u=k=2g+1=7$.   
  
   \begin{table}[h]
\begin{center}
\begin{tabular}{| l | l | l | l | }
\hline
Curve & Group & Signature & Gen. Vector \\
\hline
Wiman type I& $ {\mathbb Z}_{4g+2}$ & $(0; 4g+2, 2g+1, 2)$ & $\langle 1, 2g, 2g+1 \rangle$   \\
Wiman type II & $ {\mathbb Z}_{4g}$ &$ (0; 4g,\  4g,\  2)$ & $\langle 1, 2g-1, 2g \rangle$ \\
Accola-Maclachlan & $ {\mathbb Z}_{2g+2}$ & $(0;2g+2, 2g+2, g+1)$ & $\langle 1,\  1,\  2g \rangle$  \\
Kulkarni & $ {\mathbb Z}_{2g+2}$ & $(0;2g+2, 2g+2, g+1)$ & $\langle 1, g+2, g-1 \rangle$  \\
Wiman type III &  ${\mathbb Z}_{12}$ & $(0; 12,\ 4\ ,3)$ & $\langle 1, \ 3,\  8\rangle$\\
Klein & ${\mathbb Z}_{7}$ & $(0;7,\ 7,\ 7)$ & $\langle 1,\  2,\ 4 \rangle$   \\
\hline
\end{tabular}
\end{center}
\caption{Maximal cyclic actions on the classical curves}\label{Ta:largecyclic}
\end{table}

\section{The classical curves and edge-transitive one-vertex maps}\label{S:largeonevertex}

We have seen that a one-vertex map ${\cal M}$ with $k$ edges has $\text{Aut}({\cal M})=\langle x^p \rangle \leq {\mathbb Z}_{2k}=\langle x \rangle$, where $p$ is a divisor of $2k$.   If ${\cal M}$ is edge-transitive, $\text{Aut}({\cal M})$ must have at least $k$ elements, hence $p=1$ or $p=2$.  If $p=1$, ${\cal M}$ is regular.  If $p=2$, ${\cal M}$ is {\it strictly\,} edge transitive, i.e., edge transitive but not regular.  

The following theorem (stated  in terms of unifacial dessins) was proved in \cite{S01}, and completely characterizes regular one-vertex maps.  

\begin{theorem}\label{T:regular1vertex}  Let $\,{\cal M}$ be a regular, one-vertex map on a surface $X$   of genus   $g>0$.  Then $X$ is the Wiman curve of type I  and ${\cal M}$ has $2g+1$ edges, or $X$ is the Wiman curve of type II and $\,{\cal M}$ has $2g$ edges.  
\end{theorem}

 The regular one-vertex maps in genus $1$ are shown in Figure~\ref{F:reggenus1}.  The type I map on the left has two faces, centered at $A$  and $B$, and is the geometrization of the map in Figure~\ref{F:k=3} (b).   It lies on the elliptic curve of modulus $e^{2\pi i/3}$.  The type II map on the right has one face, centered at $A$, and lies on the elliptic curve of modulus $i$.    These curves are the natural genus $1$ analogues of the Wiman curves of type I,\, II\,,  characterized (up to conformal equivalence) by admitting an automorphism of order greater than $2$  ($6$, $4$, respectively) which {\em fix a point}.   Note that the map in Figure~\ref{F:torusmap} also lies on the type I elliptic curve.  

\begin{figure}
\begin{center}
\begin{pspicture}(-.6,-.4)(5.6,2.2)
\SpecialCoor
\degrees[6]

\rput(1;1){\pspolygon[linestyle=dashed](1;1)(1;2)(1;3)(1;4)(1;5)(1;6)}
\psdots(1;1)
\psline[linewidth=1.5 pt, arrows=->](.5,0)(1;1)
\rput(!0 3 sqrt 2 div){\psline[linewidth=1.5 pt, arrows=<-](.5,0)(1;1)}

\rput{1}(1;0)
{
\psline[linewidth=1.5 pt, arrows=->](.5,0)(1;1)
\rput(!0 3 sqrt 2 div){\psline[linewidth=1.5 pt, arrows=<-](.5,0)(1;1)}
}

\rput(1;0){
\rput{2}(1;1)
{
\psline[linewidth=1.5 pt, arrows=->](.5,0)(1;1)
\rput(!0 3 sqrt 2 div){\psline[linewidth=1.5 pt, arrows=<-](.5,0)(1;1)}
}
}
\rput(1;1){\psdots[dotstyle=o,dotscale=1](1;1)}
\rput(1;1){\psdots[dotstyle=o,dotscale=1](1;2)}
\rput(1;1){\psdots[dotstyle=o,dotscale=1](1;3)}
\rput(1;1){\psdots[dotstyle=o,dotscale=1](1;4)}
\rput(1;1){\psdots[dotstyle=o,dotscale=1](1;5)}
\rput(1;1){\psdots[dotstyle=o,dotscale=1](1;6)}
\rput(1;1){\uput[1.5](1;2){A}}
\rput(1;1){\uput[4.5](1;4){A}}
\rput(1;1){\uput[0](1;0){A}}

\rput(1;1){\uput[1.5](1;1){B}}
\rput(1;1){\uput[3](1;3){B}}
\rput(1;1){\uput[4.5](1;5){B}}


\uput{8 pt}[0.8](1;1){\tiny 0}
\uput{8 pt}[1.8](1;1){\tiny 1}
\uput{8 pt}[2.8](1;1){\tiny 2}
\uput{8 pt}[3.8](1;1){\tiny 3}
\uput{8 pt}[4.8](1;1){\tiny 4}
\uput{8 pt}[5.8](1;1){\tiny 5}

\rput(4;0)
{
\rput(1;1){\pspolygon[linestyle=dashed](.8,.8)(-.8,.8)(-.8,-.8)(.8,-.8)}
\rput(1;1){\psdots[dotstyle=o,dotscale=1](.8,.8)(-.8,.8)(-.8,-.8)(.8,-.8)}
\rput(1;1){\uput[1.5](.8,.8){A}}
\rput(1;1){\uput[1.5](-.8,.8){A}}
\rput(1;1){\uput[4.5](-.8,-.8){A}}
\rput(1;1){\uput[4.5](.8,-.8){A}}
\psdots(1;1)
\rput(1;1){\psline[linewidth=1.5 pt, arrows=->](-.8,0)(0,0)}
\rput(1;1){\psline[linewidth=1.5 pt, arrows=<-](0,0)(.8,0)}
\rput(1;1){\psline[linewidth=1.5 pt, arrows=<-](0,0)(0,-.8)}
\rput(1;1){\psline[linewidth=1.5 pt, arrows=<-](0,0)(0,.8)}
\uput{8 pt}[0.3](1;1){\tiny 0}
\uput{8 pt}[1.8](1;1){\tiny 1}
\uput{8 pt}[3.3](1;1){\tiny 2}
\uput{8 pt}[4.8](1;1){\tiny 3}
}

\end{pspicture}
\caption{Regular one-vertex maps in genus $1$}\label{F:reggenus1}
\end{center}
\end{figure}
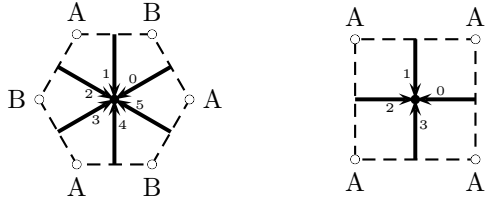

 \subsection{Strictly edge-transitive one-vertex  maps}
  If ${\cal M}$ is strictly edge-transitive with $k$ edges, the edge midpoints are permuted regularly in a $k$-cycle (with trivial isotropy subgroup) and the face centers fall into two orbits,  each having isotropy subgroup of order equal to the corresponding face-valence.   By Wiman's bound on the order of an automorphism  (or an automorphism with a fixed point if $g=1$), $k \leq 4g+2$.  $k=4g+1$ is not possible, by the Riemann-Hurwitz relation.  
  
     \begin{theorem}\label{T:edgetransaction} Let $\,{\cal M}$ be strictly edge-transitive one-vertex map on a surface $X$   of genus   $g\geq 1$.  The action of $\text{Aut}({\cal M})\simeq{\mathbb Z}_k$ has signature and generating vector
\begin{equation}\label{E:edgetransaction}
{\mathbb Z}_k \quad (0;k, l_1, l_2) \quad \langle 1, t, k-(t+1) \rangle, \quad  l_1 = {k}/{(t,k)}, \ l_2 = {k}/{(t+1,k)}, 
\end{equation}
for some $t$,  $0 < t <k/2$, $t\ne (k-1)/2$.  The notation $(\ ,\ )$ indicates the gcd (greatest common divisor) of two integers.

\end{theorem}

\begin{proof}  In the monodromy group $G_{\cal M}=\langle x, y \rangle$, the free involution $y$ commutes with $x^2$ but not with $x$, hence it is obtained from a pairing of the two $k$-cycles 
$$C_0 = (0\  2\  \dots\ 2k-2) \qquad\text{and}\qquad C_1=(1\ 3\ \dots\ 2k-1)$$
with a shift parameter $t$, $0 \leq t \leq k-1$ (cf. Section~\ref{SS:prescribedcommuters}).  If $k$ is odd, $t = (k-1)/2$ is excluded, since in that case, $y=x^k$ and the map would be  regular.  The permutation $(xy)^{-1}=yx^{-1}$ consists of cycles of at most two different lengths.  If  a cycle of $yx^{-1}$ contains an odd symbol $2c+1$, then 
$$2c+1 \overset{x^{-1}}{\quad\mapsto\quad}2c \overset{y}{\quad\mapsto\quad}2c + 2t + 1, $$
which implies that the cycle has length $l_1$ where $l_1$ is the minimal positive integer for which
$ 2l_1t \equiv 0 \pmod {2k}.$  
It follows that 
$ l_1 = {k}/{(t,k)}.$    Similarly, if a cycle contains an even symbol $2c$, then 
$$2c \overset{x^{-1}}{\quad\mapsto\quad}2c-1 \overset{y}{\quad\mapsto\quad}2c -2- 2t,$$
which implies that the cycle has length $l_2$, where $l_2$ is the minimal positive integer such that $-l_2(2+2t) \equiv 0 \pmod {2k}$.  This yields
$ l_2 = {k}/{(t+1, k)}.$  
It follows that $(xy)^{-1}$ is a product of $(t,k)$ cycles of length $l_1$ and $(t+1, k)$ cycles of length $l_2$.           The single vertex of ${\cal M}$ is fixed by $\text{Aut}({\cal M})$.  The edge midpoints are permuted regularly in a $k$-cycle, and the centers of the two face types are permuted in orbits of lengths $(t,k)$ and $(t+1, k)$, with isotropy subgroups of orders $l_1$ and $l_2$, respectively.  It follows that $\text{Aut}({\cal M})$ acts with signature $(0;k, l_1, l_2)$.   By Theorem~\ref{T:MapAuts}, $\langle x, y \rangle$ is strongly conjugate to $\langle x, y_1\rangle$, where $y_1=x^{-1}yx$.  Moreover $y_1$ is obtained from $y$ by replacing the shift parameter $t$ with the shift parameter $k-(t+1)$.  We note that since $(t+1,k) = (k-(t+1),k)$, the signature of the ${\mathbb Z}_k$ action is not changed if $t$ is replaced by $k-(t+1)$.   It follows that up to map equivalence, we may assume $t<k/2$.     If $t=0$,  the signature reduces to $(0;k,k)$ which places the map (and the ${\mathbb Z}_k$ action) on   the Riemann sphere.    The lcm of $l_1$ and $l_2$ is $k$, so a $(k, l_1, l_2)$-generating vector for ${\mathbb Z}_k$ exists and may be taken in the normal form $\langle 1, t, k-(t+1)\rangle$ (cf. \eqref{E:ZnGenVector}).  
 \end{proof}

In most cases, if a curve $X$ of genus $g>1$ has  a strictly edge-transitive, one-vertex map ${\cal M}$, then $\text{Aut}({\,\cal M})=\text{Aut}(\,X)$.  The exceptions are characterized in our final theorem, below.   If $\text{Aut}{\cal M} = \text{Aut}(\,X)$, an equation for $X$ is $w^k=z^{k/l_1}(z-1)^{k/ l_2}$, where $\text{Aut}({\cal M})\simeq {\mathbb Z}_k$  is generated by $(w,z) \mapsto (w, e^{2\pi i/k}z)$.

\begin{theorem}\label{T:main} Let ${\cal M\,}$ be a strictly edge-transitive one-vertex map with $k$ edges, where $\text{Aut}(\,{\cal M})={\mathbb Z}_k$ has signature and generating vector \eqref{E:edgetransaction}.  Let $X$ be the  canonical Riemann surface of genus $g>1$.     If $\text{Aut}(\,X) >  \text{Aut}(\,{\cal M})$,   then 
\begin{enumerate}
\item 
 $k=12$, $t=3$, and $X$ is the Wiman type III  curve;
\item 
$k=2g+1$, $t=1$, and $X$ is the Wiman type I curve;
\item 
 $k=2g+2$, $t=1$, and $X$ is the Accola-Maclachlan curve;
\item 
 $k=2g+1$, $t=\beta(1)$,   and $\text{Aut}(X) \simeq {\mathbb Z}_3\ltimes_{\beta}  {\mathbb Z}_k$ as in Lemma~\ref{L:N6}; except
\begin{itemize}
\item 
if $k=7$, $t=\beta(1)=2$,  $X$ is the Klein quartic.
\end{itemize}
\item 
 $2g+2\leq k\leq 4g$, $t=\alpha(1)$,  and $\text{Aut}(X)\simeq {\mathbb Z}_2 \ltimes_{\alpha} {\mathbb Z}_k$ as in Lemma~\ref{L:N8}; except
 \begin{itemize}
   \item 
     if $k=2g+2$, $g \equiv -1 \pmod 4$,  $\alpha(1)=g+2$,  $X$ is the Kulkarni curve; or  
   \item 
    if $k=12$ or $24$, $g=4$ or $10$, $\alpha(1)=7$ or $19$ (resp.),   $\text{Aut}(X)$ contains $ {\mathbb Z}_2 \ltimes_{\alpha} {\mathbb Z}_k$ with index $3$. 
    \end{itemize}
\end{enumerate}
In case 5, with $k=4g$ and $\alpha(1) = 2g-1$, ${\mathbb Z}_2 \ltimes_{\alpha} {\mathbb Z}_k \simeq \text{SD}_{8g}$ and $X$ is the Wiman curve of type II.
 \end{theorem}

 \begin{proof}  If $\text{Aut}({\,X}) > \text{Aut}({\,\cal M})$,  the signature in \eqref{E:edgetransaction} must coincide with one of the cyclic-admissible signatures $\sigma$ in Table~\ref{Ta:cyclicadmissible}.  For an extension of type N6, by  the Riemann-Hurwitz relation, $k=2g+1$; if $g=3$, the N6 extension is subsumed by the T1 extension, which yields the Klein quartic.  For an extension of type N8, by the Riemann-Hurwitz relation, $k= 2g(u/u-1)$.  For an abelian extension, there are two possibilities:  (i) $u=k/2$ (equivalently, $k=2g+2$) and $t=1$; or (ii) $u=k$ (equivalently, $k=2g+1$) and $t=1$.  In the first case we obtain the ${\mathbb Z}_k$ action (see Table~\ref{Ta:largecyclic})  which determines the Accola-Maclachlan curve.   (The N8 extension to ${\mathbb Z}_2 \oplus {\mathbb Z}_{2g+2}$ is subsumed by the T9 extension.)   In the second case, we obtain the cyclic action which determines the Wiman curve of type I, extending to  ${\mathbb Z}_{4g+2}$.   For a non-abelian extension of type N8, $u \leq k/2$ (otherwise the extension would be cyclic of order $2k$), hence $2g+2 \leq k \leq 4g$.   Both the lower and upper bounds, which occur when $l_2=k/2$, $l_2 = 2$, respectively, satisfy the necessary conditions on $k$ given in Lemma~\ref{L:N8} for a nonabelian, non-dihedral extension to ${\mathbb Z}_2 \ltimes_{\alpha} {\mathbb Z}_k$.  
    If $g \equiv -1 \pmod 4$, $\alpha(1)=g+2$ determines an automorphism of ${\mathbb Z}_{2g+2}$ of order $2$ since $(g+2)^2 \equiv g^2 \equiv 1 \pmod{2g+2}$.  Taking $k=2g+2$, we have  $k-(1 + \alpha(1)) = g-1$ which has has order $g+1$ (mod $k$) since $(g+1)(g-1)=(2g+2)(g-1)/2  \equiv 0 \pmod k.$  This yields the ${\mathbb Z}_k$ action which determines the Kulkarni surface.  (The N8 extension is subsumed by the T9 extension, or by the T4 extension if $g=3$).  In the last two exceptional cases ($k=12$, $24$) the N8 extension is subsumed by the T8 extension.  The Wiman type III curve arises from the T10 extension. 
 \end{proof}

Subsets of the curves mentioned in Theorem~\ref{T:main} have been characterized in ways related to our characterization.   A. Meleko\u glu and D. Singerman (\cite{MS08}, Section~6) characterized 
the Wiman curves of type I and II, and the Accola-Maclachlan curves  (all of which are hyperelliptic)  as the unique curves of genus $g>1$  admitting {\it  double-star maps}.  These are two-sheeted covers (via the hyperelliptic involution) of a one-vertex map on the sphere having only free edges (carrying a single dart).  We note that  double-star maps are the ``exceptional'' ribbon graphs of \cite{MP98} (Definition 1.9).  Meleko\u glu and  Singerman (\cite{MS08}, Theorems~8.2 and 8.3) also characterized the  Accola-Maclachlan  and Wiman type II curves  as the unique Platonic $M$- and $(M-1)$-surfaces, respectively.  These are curves of genus $g>1$ admitting an anti-conformal involution with the maximal and second-maximal ($g+1$, resp., $g$) number of fixed simple closed curves.    D. Singerman (\cite{S01}, Theorem~6.2) classified the surfaces admitting {\it uniform} strictly edge-transitive one-vertex maps; this is the special case $l_1=l_2$ in our Theorem~\ref{T:edgetransaction}.



\end{document}